\documentclass[a4paper,12pt]{amsart}  
\usepackage{amsmath,amssymb,amsthm}     
\usepackage[utf8]{inputenc}
\usepackage{enumitem}
\usepackage{color}
\usepackage[normalem]{ulem}

\theoremstyle{plain}      
 
\newtheorem{thm}{Theorem}[section]     
\newtheorem{theorem}[thm]{Theorem}

\newtheorem{lemma}[thm]{Lemma}

\theoremstyle{definition}

\newtheorem{example}[thm]{Example}

\DeclareMathAlphabet{\doba}{U}{msb}{m}{n}

\gdef\mR{\doba{R}}

\def\di{{\rm d}}

\let\<\langle 
\let\>\rangle

\newcommand{\definedas}{\mathrel{\raise.095ex\hbox{\rm :}\mkern-5.2mu=}}

\title{}

\author{Andrei Moroianu, Mihaela Pilca}

\address{Andrei Moroianu \\ Université Paris-Saclay, CNRS,  Laboratoire de mathématiques d'Orsay, 91405, Orsay, France}
\email{andrei.moroianu@math.cnrs.fr}

\address{Mihaela Pilca\\Fakult\"at f\"ur Mathematik\\
Universit\"at Regensburg\\Universit\"atsstr. 31 
D-93040 Regensburg, Germany}
\email{mihaela.pilca@mathematik.uni-regensburg.de}



\begin{document}   

\title{Metric connections with parallel  twistor-free torsion}

\begin{abstract}
The torsion of every metric connection on a Riemannian manifold has three components: one totally skew-symmetric, one of vectorial type, and one of twistorial type. In this paper we classify complete simply connected Riemannian manifolds carrying a metric connection whose torsion is parallel, has non-zero vectorial component and vanishing twistorial component.

\end{abstract}
\maketitle

\section{Introduction}

The main tool of Riemannian geometry is the Levi-Civita connection, which is the unique torsion-free metric connection on a given Riemannian manifold. However, on Riemannian manifolds carrying additional structures ({\em e.g.} Sasakian \cite{ad}, \cite{friedrich2}, almost Hermitian \cite{andrei}, \cite{schoeman}, hyperhermitian \cite{alex2}, $\mathrm{G}_2$ \cite{friedrich1}, Spin(7) \cite{ivanov}, homogeneous \cite{ambrose}, etc.), metric connections with torsion are more adapted in order to understand the underlying geometries. 

In most of the aforementioned cases, the torsion of the relevant metric connection is totally skew-symmetric and parallel. In \cite{c-swann}, Cleyton and Swann obtained the classification of metric connections with parallel skew-symmetric torsion whose holonomy representation is irreducible. They show that in this case, the manifold is either naturally reductive homogeneous, or nearly Kähler in dimension 6, or nearly parallel G$_2$ in dimension 7. 

However, since no analogue of the de Rham decomposition theorem holds for connections with torsion, the reducible case is more involved and not completely classified. A systematic study of this problem in the reducible case has been started recently in \cite{cms}, where it was shown that every Riemannian manifold carrying a metric connection with parallel skew-symmetric torsion is locally a Riemannian submersion with totally geodesic naturally reductive homogeneous fibers over a lower-dimensional manifold carrying a principal bundle with parallel curvature. 

The case of metric connections with torsion of vectorial type was considered by Agricola and Kraus in \cite{a}, where it is noticed that the condition of having parallel torsion is too restrictive, and is relaxed by asking that the corresponding 1-form is closed.

In the present paper we study a more general problem, by asking for the torsion of the metric connection to be parallel, while allowing it to have both a skew-symmetric and a vectorial component. Since the third component of the torsion is called twistorial component, such connections will be referred to as having twistor-free torsion.  Surprisingly, it turns out that if the vectorial component is non-zero, then the problem is completely solvable. 

The rough idea of the classification is as follows. In Theorem \ref{thred} we show that the presence of the vectorial component forces the manifold to decompose as a warped product with fiber $\mathbb{R}$, with explicit warping function, and with basis carrying a parallel 3-form $\tau$ satisfying the condition $(\tau_X)_*\tau=0$ for every basic tangent vector $X$ (here $\tau_X$ denotes the skew-symmetric endomorphism corresponding to $X\lrcorner\tau$ via the metric on the basis, and acting as a derivation on the exterior bundle). 

The condition $(\tau_X)_*\tau=0$ can be interpreted as the Jacobi identity for the bracket defined on each tangent space by $[X,Y]:=\tau_XY$, and induces a parallel compact type Lie algebra structure on the tangent bundle. We then show in Theorem \ref{thdec} that such structures decompose as products of irreducible bricks of four types: symmetric spaces of type II or IV, 3-dimensional Riemannian manifolds, simple Lie algebras of compact type with bi-invariant metric, and arbitrary Riemannian manifolds with abelian Lie algebra structure. 

Using these results, the complete classification of metric connections with parallel twistor-free torsion is given in Theorem \ref{thmain}.

{\bf Acknowledgments.} This work was supported by the Procope Project No. 57445459 (Germany) /  42513ZJ (France).

\section{Preliminaries}

Let us first introduce some notation and conventions used in this paper. On a  Riemannian manifold $(M,g)$ vectors and $1$-forms or $g$-skew-symmetric endomorphisms and $2$-forms will be as usually identified via the metric $g$. A $3$-form $\tau$ on $M$ will be identified with a tensor of type $(2,1)$ as follows:
\[\tau(X,Y,Z)=g(\tau_X Y, Z), \quad \forall X,Y,Z\in\Gamma(\mathrm{T}M).\]
In this way, the $2$-form $X\lrcorner\tau$ is identified with the skew-symmetric endomorphism $\tau_X$ for every tangent vector $X$. The kernel of a $3$-form $\tau$ at a point $x\in M$ is defined as follows: 
$$\mathrm{ker}(\tau):=\{X\in \mathrm{T}_x M\, \,|\, \tau_{X}=0\}.$$

For every $k\ge 0$, a skew-symmetric endomorphism $A$ of $\mathrm{T}M$ acts as a derivation on the bundle of exterior $k$-forms by the formula
\begin{equation}\label{action}A_*\sigma:=\sum_i Ae_i\wedge e_i\lrcorner\sigma,\qquad\forall \sigma\in \Gamma(\Omega^kM),
\end{equation}
where $\{e_i\}_i$ is a local orthonormal basis of $\mathrm{T}M$. For later use, note that if $k=2$ and $\sigma$ is identified with a skew-symmetric endomorphism via the metric, then $A_*\sigma$ is the $2$-form corresponding to the commutator $[A,\sigma]$.

Let $\nabla^g$ denote the Levi-Civita connection of $g$. Every other metric connection $\nabla$ on $(M,g)$ can be written as 
$$\nabla_X=\nabla^g_X+T_X,\qquad\forall X\in \mathrm{T}M,$$
where $T\in \Gamma(\Lambda^1M\otimes\mathrm{End}^-(\mathrm{T}M))$ is a 1-form on $M$ with values in the $g$-skew-symmetric endomorphisms of $\mathrm{T}M$. The tensor $T$ can be identified with the torsion $\widetilde T$ of $\nabla$ via the isomorphism 
$$\Lambda^1M\otimes\mathrm{End}^-(\mathrm{T}M)\to \Lambda^2 M\otimes \mathrm{T}M,\qquad T\mapsto \widetilde T, 
$$
with $\widetilde T(X,Y):=T_XY-T_YX$. 

Recall that the tensor product of the $\mathrm{SO}(n)$ representations $\mathbb{R}^n$ and $\Lambda^2\mathbb{R}^n$ decomposes as 
$$\mathbb{R}^n\otimes\Lambda^2\mathbb{R}^n=\mathbb{R}^n\oplus\mathcal{T}^{2,1}\mathbb{R}^n\oplus \Lambda^3\mathbb{R}^n,$$
where $\mathcal{T}^{2,1}$ is the Cartan summand and is generated by elements of the form $v\otimes \omega\in \mathbb{R}^n\otimes\Lambda^2\mathbb{R}^n$ which satisfy $v\wedge\omega=0$ and $v\lrcorner\omega=0$. The inclusions of $\mathbb{R}^n$ and $\Lambda^3\mathbb{R}^n$ into $\mathbb{R}^n\otimes\Lambda^2\mathbb{R}^n$ are given by 
$$\xi\mapsto \sum_ie_i\otimes (e_i\wedge\xi),\qquad \tau\mapsto \sum_i e_i\otimes (e_i\lrcorner\tau),$$
where $\{e_i\}_i$ is any orthonormal basis of $\mathbb{R}^n$.
Correspondingly, on every Riemannian manifold we have the decomposition
$$\Lambda^{1}M\otimes \Lambda^{2}M=\Lambda^{1}M\oplus \Lambda^{2,1}\mathrm{T}M\oplus \Lambda^{3}M.$$ 
Therefore, the torsion of a metric connection $\nabla=\nabla^g+T$, identified with the tensor $T\in \Gamma(\Lambda^{1}M\otimes \Lambda^{2}M)$,
can be decomposed as $T=T_1+T_2+T_3$ with: 
$$T_1\in \mathcal{T}_1(M):=\Gamma(\Lambda^{1}M),\qquad T_2\in \mathcal{T}_2(M):=\Gamma(\Lambda^{2,1}M),\qquad T_3\in \mathcal{T}_3(M):=\Gamma(\Lambda^{3}M). $$

By analogy with the Gray-Hervella classification of almost Hermitian structures, we will say that the torsion of a metric connection $\nabla=\nabla^g+T$ is of type 
$\mathcal{T}_i$ for $i\in\{1,2,3\}$ if $T_i\ne 0$ and $T_j=0$ for $j\ne i$. Similarly, for distinct $i,j\in\{1,2,3\}$ we say that the torsion of $\nabla$ is of type
$\mathcal{T}_i\oplus\mathcal{T}_j$ if $T_i\ne0$, $T_j\ne 0$ and $T_k=0$ for $k$ different from $i,j$.

Recall that the twistor operator on Riemannian manifolds acting on $2$-forms is the projection of their covariant derivative from $\Gamma(\Lambda^{1}M\otimes \Lambda^{2}M)$ onto $\Gamma(\Lambda^{2,1}\mathrm{T}M)=\mathcal{T}_2(M)$. We will thus call the $\mathcal{T}_2(M)$-component of a tensor in $\Gamma(\Lambda^{1}M\otimes \Lambda^{2}M)$ its twistorial component. Correspondingly, a metric connection is said to have twistor-free torsion if its torsion is of type $\mathcal{T}_1\oplus\mathcal{T}_3$ and twistor-like torsion if its torsion is of type $\mathcal{T}_2$.

The aim of this paper is the classification of complete simply connected Riemannian manifolds admitting a metric connection $\nabla$ whose torsion is twistor-free and $\nabla$-parallel.

\section{Metric connections with parallel twistor-free torsion}

Let $(M,g_M)$ be a complete simply connected $n$-dimensional Riemannian manifold with Levi-Civita connection $\nabla^{g_M}$, which is endowed with a  metric connection $\nabla$ with twistor-free torsion. Assume moreover that the torsion is also $\nabla$-parallel. By the above considerations, there exists a non-zero vector field $\xi\in\Gamma(\mathrm{T}M)$ and a non-zero $3$-form $\nu\in\Omega^3(M)$, such that for all vector fields $X\in\Gamma(\mathrm{T}M)$ the following identity holds:
\begin{equation}\label{defnabla}
\nabla_X=\nabla^{g_M}_X+X\wedge \xi+\nu_X.
\end{equation}
The fact that the torsion of $\nabla$ is $\nabla$-parallel is clearly equivalent to $\nabla \xi=0$ and $\nabla \nu=0$. 

Our first aim is to reduce the problem to the study of a metric connection with parallel skew-symmetric torsion satisfying some further algebraic constraint on an $(n-1)$-dimensional Riemannian manifold. We start with the following technical result:

\begin{lemma}\label{elemprop}
With the above notation, the following identities hold:
$$\nu_{\xi}=0, \quad \di\xi=0, \quad \sum_{i=1}^{n}\nu_{e_i}\wedge \nu_{e_i}=0, \quad \di\nu=3\xi\wedge\nu,$$
where $\{e_i\}_i$ is a local orthonormal basis of $\mathrm{T}M$.
\end{lemma}	

\begin{proof}
Up to rescaling the metric $g$, we may assume that the parallel vector field $\xi$ has constant length equal to $1$.
	Applying \eqref{defnabla} to $\xi$ yields $\nabla^{g_M}_X\xi=X-\<X,\xi\>\xi-\nu_X\xi$ for every $X\in \mathrm{T}M$. Hence we obtain:
	$$\di \xi=\sum_i{e_i}\wedge\nabla^{g_M}_{e_i}\xi=\sum_i e_i\wedge(e_i-\<e_i,\xi\>\xi-\nu_{e_i}\xi)=\sum_i e_i\wedge (e_i\lrcorner\nu_{\xi})=2\nu_\xi.$$
	We compute now the covariant derivative of $\nu_\xi$ using \eqref{action}, \eqref{defnabla} and the fact that both $\xi$ and $\nu$ are $\nabla$-parallel as follows:
	$$\nabla^{g_M}_X\nu_\xi=-\sum_i(X\wedge \xi)e_i\wedge e_i\lrcorner\nu_\xi-\sum_i\nu_Xe_i\wedge e_i\lrcorner\nu_\xi=-\xi\wedge \nu_\xi X+\sum_i\nu_{e_i}X\wedge \nu_\xi e_i.$$
	Taking the wedge product with $X$ and summing over $X=e_i$, we obtain:
	\begin{equation*}
	\begin{split}
	\di\nu_\xi&=\sum_i e_i\wedge \nabla^{g_M}_{e_i}\nu_\xi=\sum_i e_i\wedge\left(-\xi\wedge \nu_\xi e_i+\sum_j\nu_{e_j}e_i\wedge \nu_\xi e_j\right)=2\xi\wedge \nu_\xi-2\sum_i \nu_{e_i}\wedge \nu_{e_i}\xi\\
	&=2\xi\wedge \nu_\xi-\xi\lrcorner\sum_i \nu_{e_i}\wedge \nu_{e_i}.
	\end{split}
	\end{equation*}
On the other hand, since $\di\xi=2\nu_\xi$, it follows that $\di\nu_\xi=0$. The above computation then implies that 
$$2\xi\wedge \nu_\xi=\xi\lrcorner\sum_i \nu_{e_i}\wedge \nu_{e_i}.$$
Taking a further interior product with $\xi$ and using that $\xi$ is nowhere vanishing (being non-zero and $\nabla$-parallel) yields $\nu_\xi=0$, hence also 
$\di\xi=2\nu_\xi=0$ and 
\begin{equation}\label{alpha}\xi\lrcorner\sum_i \nu_{e_i}\wedge \nu_{e_i}=0.
\end{equation} 
We further compute the covariant derivative of $\nu$ and its exterior differential:
\begin{eqnarray*}
\nabla^{g_M}_X\nu&=&-(X\wedge \xi)_*\nu-(\nu_X)_*\nu=-\sum_i(X\wedge \xi)e_i\wedge e_i\lrcorner\nu-\nu_X e_i\wedge e_i\lrcorner \nu\\
&=&-\xi\wedge\nu_X+\sum_i\nu_{e_i}X\wedge \nu_{e_i},\\
\di\nu&=&\sum_i e_i\wedge\nabla^{g_M}_{e_i}\nu=-\sum_i e_i\wedge\xi\wedge\nu_{e_i}+\sum_{i,j}e_i\wedge\nu_{e_j}e_i\wedge \nu_{e_j}=3\xi\wedge\nu+2\alpha,
\end{eqnarray*}
where $\displaystyle\alpha:=\sum_{i} \nu_{e_i}\wedge \nu_{e_i}$.
It remains to show that $\alpha$ vanishes. The covariant derivative of $\alpha$ and its exterior differential are obtained as follows, where we assume furthermore that the orthonormal basis $\{e_i\}_i$ is parallel with respect to the Levi-Civita connection $\nabla^{g_M}$ at the point where the computation is done:
\begin{equation*}
\begin{split}
\nabla^{g_M}_X\alpha&=\sum_i \nabla^{g_M}_X(\nu_{e_i}\wedge \nu_{e_i})=\sum_i 2\nu_{e_i}\wedge\nabla^{g_M}_X\nu_{e_i}=2\sum_i \nu_{e_i}\wedge\nabla^{g_M}_X\nu_{e_i}\\
&=-2\sum_{i,j} \nu_{e_i}\wedge(X\wedge\xi)(e_j)\wedge e_j\lrcorner\nu_{e_i}-2\sum_{i,j}\nu_{e_i}\wedge\nu_Xe_j\wedge e_j\lrcorner \nu_{e_i}\\
&=-2\sum_i \nu_{e_i}\wedge \xi \wedge \nu_{e_i} X+2\sum_i \nu_{e_i}\wedge X \wedge \nu_{e_i} \xi+2\sum_{i,j}\nu_{e_j}X\wedge  \nu_{e_i} e_j\wedge \nu_{e_i}\\
&=-2\sum_i \xi \wedge \nu_{e_i} X\wedge \nu_{e_i}+2\sum_{i,j}\nu_{e_j}X\wedge  \nu_{e_i} e_j\wedge \nu_{e_i},
\end{split}
\end{equation*}
\begin{equation*}
\begin{split}
\di\alpha&=\sum_i e_i\wedge\nabla^{g_M}_{e_i}\alpha=-2\sum_{ij} e_i\wedge \xi\wedge \nu_{e_j}e_i\wedge \nu_{e_j}+2\sum_{i,j,k}e_i\wedge \nu_{e_j}e_i\wedge  \nu_{e_k} e_j\wedge \nu_{e_k}\\
&=4\sum_j \xi\wedge \nu_{e_j}\wedge \nu_{e_j}+4\sum_{j,k}\nu_{e_j}\wedge \nu_{e_k}e_j\wedge \nu_{e_k}=4\xi\wedge \alpha.
\end{split}
\end{equation*}
We now obtain:
$$0=\di^2\nu=\di(3\xi\wedge \nu+2\alpha)=-3\xi\wedge \di\nu+2\di\alpha=-6\xi\wedge \alpha+8\xi\wedge \alpha=2\xi\wedge \alpha,$$
showing that $\alpha=0$, because $\xi\lrcorner\alpha=\displaystyle\xi\lrcorner\sum_i\nu_{e_i}\wedge \nu_{e_i}=0$ by \eqref{alpha}. This concludes the proof of the lemma.
\end{proof}

We can now state our reduction result:

\begin{theorem}\label{thred}
	A complete simply connected Riemannian manifold $(M,g_M)$ carries a metric connection with parallel twistor-free torsion if and only if $(M, g_M)$ is homothetic to a warped product $(N\times \mR, e^{2t}g_N+\di t^2)$, where  $(N, g_N)$ is a complete simply connected Riemannian manifold carrying a parallel $3$-form $\tau\in\Omega^3(N)$ which satisfies $\tau_X\tau=0$, for all $X\in \mathrm{T}N$.
\end{theorem}

\begin{proof}
Let $\nabla$ be a metric connection on $(M, g_M)$ with parallel twistor-free torsion given as:
\begin{equation}\label{defconn}
\nabla_X=\nabla^{g_M}_X+X\wedge \xi+\nu_X,
\end{equation}
with $\nabla \xi=0$ and $\nabla \nu=0$. After rescaling the metric if necessary, one can assume that $\xi$ has unit length. Let $\eta:=g_M(\xi,\cdot)$ denote the metric dual of $\xi$. By Lemma \ref{elemprop} one has $\di\eta=0=\nu_\xi$, so applying \eqref{defconn} to $\xi$ yields for all vector fields $X$ on $M$:
\begin{equation}\label{formxi}
\nabla^{g_M}_X\xi=X-\eta(X)\xi.
\end{equation}

By assumption, $M$ is simply connected, so there exists a function $t:M\to\mathbb{R}$ with $\di t=\eta$. We denote by $N:=t^{-1}(0)$ the level hypersurface of $t$ at $0$, endowed with the induced Riemannian metric $g_N$. Consider the new metric on $M$ given by $\tilde g:=e^{-2t}g_M$.
From the standard conformal change formulas we have for any vector fields $X,Y$ on $M$:
\begin{eqnarray*}\nabla^{\tilde g}_XY&=&\nabla^{g_M}_XY-X(t)Y-Y(t)X+g_M(X,Y)\mathrm{grad}^{g_M}t\\
&=&\nabla^{g_M}_XY-\eta(X)Y-\eta(Y)X+g_M(X,Y)\xi.\end{eqnarray*}
Applying this formula to $Y:=e^{t}\xi$ and using \eqref{formxi} yields for every vector field $X$:
\begin{eqnarray*} \nabla^{\tilde g}_X(e^t\xi)&=&\nabla^{g_M}_X(e^t\xi)-\eta(X)e^t\xi-e^tX+g_M(X,e^t\xi)\xi\\&=&e^t(X(t)\xi+\nabla^{g_M}_X\xi-\eta(X)\xi-X+\eta(X)\xi)=0. \end{eqnarray*}

Consequently, the vector field $e^t\xi$ is parallel and has unit length on $(M,e^{-2t}g_M)$. This shows that $(M,e^{-2t}g_M)$ is globally isometric to $(N,g_N)\times (\mathbb{R},\di s^2)$, where $s$ is determined by the fact that $\di s$ is the metric dual of $e^t\xi$ with respect to $\tilde g$, {\em i.e.} $\di s=\tilde g(e^t\xi,\cdot)=e^{-t}\eta=e^{-t}\di t$. This shows that $M=N\times \mathbb{R}$ and $g_M=e^{2t}(g_N+\di s^2)= e^{2t}g_N+\di t^2$.

Using Cartan's formula and Lemma~\ref{elemprop} we compute:
\[\mathcal{L}_{\xi} \nu=\di\nu_\xi+\xi\lrcorner d\nu=3\xi\lrcorner(\xi\wedge \nu)=3\nu.\]
Since $\xi=\frac{\partial}{\partial t}$, there exists $\tau\in\Omega^3(N)$ such that $\nu=e^{3t}\tau$. 

Let $\{e_i\}_i$ be an orthonormal basis of $\mathrm{T}_xM$, for some $x\in M$. By Lemma \ref{elemprop}, for every $X\in \mathrm{T}_xM$ we have
\begin{equation}\label{nu}(\nu_X)_*\nu=\sum_i\nu_Xe_i\wedge\nu_{e_i}=-\sum_i\nu_{e_i}X\wedge\nu_{e_i}=-\frac12 X\lrcorner\sum_i\nu_{e_i}\wedge\nu_{e_i}=0.\end{equation}
This shows in particular that $(\tau_X)_*\tau=0$, for all $X\in \mathrm{T}N$. 

It remains to check that $\nabla^{g_N}\tau=0$. Let $x\in N$ and $X,Y,Z,W\in \Gamma(\mathrm{T}N)$ which are $\nabla^{g_N}$-parallel at $x$. We extend these vector fields to $M$ arbitrarily.
By \eqref{formxi}, the second fundamental form of the hypersurface $N\subset M$ is the identity of $\mathrm{T}N$, so at $x$ we have $\nabla^{g_M}_XY=-g_M(X,Y)\xi$. Using \eqref{nu} together with the fact that $\xi\lrcorner\nu=0$ and $\nabla\nu=0$, we can compute at $x$:
\begin{eqnarray*}(\nabla^{g_N}_X\tau)(Y,Z,W)&=&X(\tau(Y,Z,W))=X(\nu(Y,Z,W))=(\nabla^{g_M}_X\nu)(Y,Z,W)\\
&=&-((X\wedge\xi+\nu_X)_*\nu)(Y,Z,W)=(\xi\wedge\nu_X-(\nu_X)_*\nu)(Y,Z,W)=0.
\end{eqnarray*}

The converse statement follows in a straightforward way, by reversing the above computation.
\end{proof}

\section{Examples on symmetric spaces}

In this section we investigate the condition given in the conclusion of Theorem \ref{thred} in the framework of symmetric spaces. We first show that examples of  complete simply connected Riemannian manifolds carrying a non-zero parallel $3$-form $\tau$ satisfying $\tau_X\tau=0$ for all tangent vectors $X$ are provided by symmetric spaces of type II and IV, and then we prove that these are the only examples in the irreducible case. 

Let $G/H$ be an irreducible Riemannian symmetric space, where $G$ is a simply connected Lie group and $H$ is a compact subgroup of $G$. By definition, there is a decomposition of the Lie algebra of $G$ as: $\mathfrak{g}=\mathfrak{h}\oplus\mathfrak{m}$, where $\mathfrak{h}$ is the Lie algebra of $H$, 
\begin{equation}\label{mh}[\mathfrak{m}, \mathfrak{m}]\subseteq\mathfrak{h}, \qquad [\mathfrak{h}, \mathfrak{m}]\subseteq\mathfrak{m},
\end{equation} 
and $\mathfrak{m}$ carries an $\mathfrak{h}$-invariant scalar product $\<\cdot, \cdot\>_\mathfrak{m}$.  We denote by $\lambda\colon\mathfrak{h}\to \mathrm{so}(\mathfrak{m})\simeq\Lambda^2\mathfrak{m}$ the differential at the identity of the isotropy representation of $H$. 

Let $B_{\mathfrak{g}}\in \mathrm{Sym}^2(\mathfrak{g}^*)$ denote the Killing form of $\mathfrak{g}$, defined by $B_{\mathfrak{g}}(X,Y):=\mathrm{Tr}(\mathrm{ad}(X)\circ \mathrm{ad}(Y))$ for every $X,Y$ in $\mathfrak{g}$. Clearly $B_{\mathfrak{g}}(X,Y)=0$ for every $X\in\mathfrak{h}$ and $Y\in \mathfrak{m}$ (since by \eqref{mh}, $\mathrm{ad}(X)\circ \mathrm{ad}(Y)$ maps  $\mathfrak{m}$ to  $\mathfrak{h}$ and  $\mathfrak{h}$ to  $\mathfrak{m}$).
Recall that $G/H$ is called of compact type if $B_{\mathfrak{g}}$ is positive definite on $\mathfrak{m}$ and of non-compact type if $B_{\mathfrak{g}}$ is negative definite on $\mathfrak{m}$. Every irreducible Riemannian symmetric space is either of compact type, or of non-compact type \cite[Prop. 7.4]{kn2}.

\begin{lemma}\label{eps}
The scalar product $\<\cdot, \cdot\>_\mathfrak{m}$ on $\mathfrak{m}$ can be extended to an $\mathfrak{h}$-invariant scalar product $\<\cdot, \cdot\>_\mathfrak{g}$ on $\mathfrak{g}$, such that $\mathfrak{m}$ and $\mathfrak{h}$ are orthogonal and such that for all $X, Y\in\mathfrak{m}$ and $A\in\mathfrak{h}$ the following identity holds:
\begin{equation}\label{epsilon}
\<[X, Y], A\>_\mathfrak{g}=\varepsilon \<Y, [X, A]\>_\mathfrak{g},
\end{equation}
where $\varepsilon=-1$, if $G/H$ is of compact type, and $\varepsilon=1$, if $G/H$ is of non-compact type. 
\end{lemma}

\begin{proof} If  $G/H$ is of compact type, $B_{\mathfrak{g}}$ is negative definite on $\mathfrak{g}$, and  $\mathrm{ad}(\mathfrak{g})$-invariant. By Schur's lemma, there exists a positive constant $\lambda$ such that $B_{\mathfrak{g}}|_{\mathfrak{m}}=-\lambda\<\cdot, \cdot\>_\mathfrak{m}$. Then $\<\cdot, \cdot\>_\mathfrak{g}:=-\frac1\lambda B_{\mathfrak{g}}$ is an $\mathfrak{h}$-invariant scalar product extending $\<\cdot, \cdot\>_\mathfrak{m}$ to $\mathfrak{g}$, making $\mathfrak{m}$ and $\mathfrak{h}$ orthogonal, and satisfying~\eqref{epsilon} for $\varepsilon=-1$ thanks to the $\mathrm{ad}(\mathfrak{g})$-invariance of $B_{\mathfrak{g}}$.

If $G/H$ is of non-compact type, then the real subspace $\mathfrak{g}':=\mathfrak{h}\oplus i\mathfrak{m}$ of $\mathfrak{g}^\mathbb{C}$ is a Lie subalgebra of compact type, and the above splitting makes $(\mathfrak{g}',\mathfrak{h})$ a symmetric pair of compact type. By the first part of the proof, the $\mathfrak{h}$-invariant scalar product on $i\mathfrak{m}$ defined by $$\<iX,iY\>_{i\mathfrak{m}}:=\<X,Y\>_\mathfrak{m}$$ extends to an $\mathfrak{h}$-invariant scalar product $\<\cdot, \cdot\>_{\mathfrak{g}'}$ on $\mathfrak{g}'$, making $i\mathfrak{m}$ and $\mathfrak{h}$ orthogonal, and satisfying \eqref{epsilon} for $\varepsilon=-1$. If $\<\cdot, \cdot\>_{\mathfrak{h}}$ denotes the restriction of $\<\cdot, \cdot\>_{\mathfrak{g}'}$ to $\mathfrak{h}$, then $\<\cdot, \cdot\>_{\mathfrak{g}}:=\<\cdot, \cdot\>_{\mathfrak{h}}+\<\cdot, \cdot\>_{\mathfrak{m}}$ is an $\mathfrak{h}$-invariant scalar product extending $\<\cdot, \cdot\>_\mathfrak{m}$ to $\mathfrak{g}$, making $\mathfrak{m}$ and $\mathfrak{h}$ orthogonal, and for every  $X, Y\in\mathfrak{m}$ and $A\in\mathfrak{h}$ we have 
\begin{eqnarray*}\<[X, Y], A\>_\mathfrak{g}&=&\<[X, Y], A\>_\mathfrak{h}=\<[X, Y], A\>_{\mathfrak{g}'}=-\<[iX, iY], A\>_{\mathfrak{g}'}=\< iY,[iX, A]\>_{\mathfrak{g}'}\\
&=&\< iY,i[X, A]\>_{i\mathfrak{m}}=\< Y,[X, A]\>_{\mathfrak{m}}=\< Y,[X, A]\>_{\mathfrak{g}}.
\end{eqnarray*}
\end{proof}

Recall that the canonical $3$-form of a Lie algebra $\mathfrak{h}$ of compact type is defined as follows: 
\begin{equation}\label{defcanform}
\omega(X,Y,Z):=B_\mathfrak{h}([X, Y], Z), \quad \forall X, Y, Z\in\mathfrak{h},
\end{equation} 
where $B_\mathfrak{h}$ denotes the Killing form of $\mathfrak{h}$. Since $B_\mathfrak{h}$ is $\mathrm{ad}(\mathfrak{h})$-invariant, the Jacobi identity shows that $\omega$ is $\mathrm{ad}(\mathfrak{h})$-invariant and satisfies $(\omega_X)_*\omega=0$, for all $X\in\mathfrak{h}$.

\begin{example}\label{exsym}
Let $G/H$ be either an irreducible simply connected symmetric space of type II,  \emph{i.e.} $G:=H\times H$, $H$ is embedded diagonally in $G$  and $\mathfrak{g}=\Delta^+\oplus\Delta^{-}$, with $\mathfrak{h}\simeq\Delta^+:=\{(X,X)\, |\, X\in \mathfrak{h}\}$ and $\mathfrak{m}=\Delta^-:=\{(X,-X)\, |\, X\in \mathfrak{h}\}$, or an irreducible simply connected symmetric space of type IV, \emph{i.e.} $G:=H^{\mathbb{C}}$ and $\mathfrak{g}=\mathfrak{h}\oplus i\mathfrak{h}$. 

In both cases, there exists an $\mathfrak{h}$-invariant isomorphism, $\psi\colon \mathfrak{h}\to \mathfrak{m}$, defined in the type II case by $\psi(X):=(X, -X)$ and in the type IV case by $\psi(X):=iX$. The pull-back of the canonical form $\omega$ of $\mathfrak{h}$ through $\psi^{-1}$ thus defines an $\mathfrak{h}$-invariant $3$-form $\tau$ on $\mathfrak{m}$, which also satisfies  $(\tau_X)_*\tau=0$, for all $X\in\mathfrak{m}$. Hence, in both cases, $\tau$ defines a parallel $3$-form on $G/H$, also denoted by $\tau$, such that $(\tau_X)_*\tau=0$, for all tangent vectors $X\in\Gamma(G/H)$.
\end{example}

Conversely, we show:

\begin{theorem}\label{symmspaces} 
		Let $G/H$ be an irreducible Riemannian symmetric space carrying a parallel non-zero $3$-form $\tau$ which satisfies $(\tau_X)_*\tau=0$, for all tangent vectors $X$. Then $G/H$ is an  irreducible symmetric space of type II or IV and $\tau$ is, up to a constant multiple, equal to the above constructed $3$-form.
\end{theorem}		

\begin{proof}
	Let $\mathfrak{g}=\mathfrak{h}\oplus\mathfrak{m}$ be the decomposition of the Lie algebra of $G$ satisfying \eqref{mh}, and let $\<\cdot, \cdot\>_\mathfrak{m}$ denote the $\mathfrak{h}$-invariant scalar product on $\mathfrak{m}$ induced by the Riemannian metric on $G/H$. We extend it to an $\mathfrak{h}$-invariant scalar product $\<\cdot, \cdot\>_\mathfrak{g}$ on $\mathfrak{g}$ satisfying \eqref{epsilon} by Lemma \ref{eps}. We identify vectors and covectors on $\mathfrak{g}$ using this scalar product. 
	
	A parallel $3$-form $\tau$ on the symmetric space $G/H$ is determined by an $\mathfrak{h}$-invariant  $3$-form in $ \Lambda^3 \mathfrak{m}$, which we further denote by $\tau$. Note that the $\mathfrak{h}$-invariance of $\tau$, viewed as a linear map $\tau\colon\Lambda^1\mathfrak{m}\to\Lambda^2\mathfrak{m}$, reads
	\begin{equation}\label{inv}
	[\lambda(A),\tau(X)]=\tau([A,X]),\qquad\forall \ A\in\mathfrak{h},\ X\in\mathfrak{m},
	\end{equation}
	where the first bracket is the commutator in $\mathrm{so}(\mathfrak{m})\simeq\Lambda^2\mathfrak{m}$.
	
		Applying Lemma~\ref{alglemma} (proved in the Appendix) to the $\mathfrak{h}$-representation $V:=\mathfrak{m}$, we obtain the inclusion $\lambda(\mathfrak{h})\subseteq\tau(\mathfrak{m})\subseteq\Lambda^2\mathfrak{m}$. The fact that $G/H$ is irreducible and $\tau$ is parallel imply that $\tau_X\ne 0$, for all $X\ne 0$. Thus $\tau\colon\Lambda^1\mathfrak{m}\to\Lambda^2\mathfrak{m}$ is injective, so one can define an injective $\mathfrak{h}$-invariant map  $\varphi\colon\mathfrak{h}\to \mathfrak{m}, \, \varphi:=\tau^{-1}\circ\lambda$, where $\tau^{-1}\colon \tau(\mathfrak{m})\to \mathfrak{m}$. Since $\mathfrak{m}$ is an irreducible $\mathfrak{h}$ representation, $\varphi$ is bijective, so in particular $\mathfrak{h}$ is simple.
		
		 By Schur's Lemma, and the $\mathfrak{h}$-invariance of $\varphi$, the pull-back of $\<\cdot, \cdot\>_\mathfrak{m}$ through $\varphi$ is a constant multiple of the restriction to $\mathfrak{h}$ of $\<\cdot, \cdot\>_\mathfrak{g}$. Hence, up to rescaling $\tau$, we may assume that $\varphi$ is an isometry. 
		 
		We claim that $\varphi$ fulfills the following identities for all $A,B\in\mathfrak{h}$:
		\begin{equation}\label{id1}
		[\varphi(A), B]=\varphi([A,B])=	[A, \varphi(B)],
		\end{equation}
		\begin{equation}\label{ident}
		[\varphi(A), \varphi(B)]=-\varepsilon[A,B],
		\end{equation}
		where 
		$\varepsilon=-1$, if $G/H$ is of compact type, and $\varepsilon=1$, if $G/H$ is of non-compact type. 
		
		Using the $\mathfrak{h}$-invariance of $\tau$ given by \eqref{inv} we compute:
\begin{eqnarray*}\tau(\varphi([A,B]))&=&\lambda([A,B])=[\lambda(A),\lambda(B)]=[\tau(\varphi(A)),\lambda(B)]=-[\lambda(B),\tau(\varphi(A))]\\&=&-\tau([B,\varphi(A)])=\tau([\varphi(A),B]),
\end{eqnarray*}
so the injectivity of $\tau$ yields \eqref{id1}.
		The identity \eqref{ident} is a consequence of  \eqref{epsilon} and \eqref{id1} together with the following computation, which holds for all $A,B, C\in\mathfrak{h}$:
		\begin{eqnarray*}
\< [\varphi(A),\varphi(B)], C\>_\mathfrak{g}&=&\varepsilon	\<\varphi(B), [\varphi(A), C]\>_\mathfrak{g}=\varepsilon	\<\varphi(B), \varphi([A, C])\>_\mathfrak{g}=\varepsilon	\<B, [A, C]\>_\mathfrak{g}\\&=&-\varepsilon\<[A,B], C\>_\mathfrak{g}.
		\end{eqnarray*}
		
		Let us now define the following maps. If $\varepsilon=-1$, then
		\[\Psi_{-}\colon \mathfrak{h}\oplus \mathfrak{h}\to \mathfrak{h}\oplus \mathfrak{m} , \quad \Psi_-(A, B):=\frac{1}{2}\left(A+B+\varphi(A-B)\right),\]
		and if $\varepsilon=1$, then
			\[\Psi_+\colon \mathfrak{h}^{\mathbb{C}}\to \mathfrak{h}\oplus \mathfrak{m} , \quad \Psi_+(A+iB):=A+\varphi(B).\]
			
We claim that $\Psi_-$ and $\Psi_+$  are isomorphisms of Lie algebras. Using \eqref{id1} and \eqref{ident} we compute for all $A_1, A_2, B_1, B_2\in\mathfrak{h}$:
		\begin{equation*}
		\begin{split}
		[\Psi_-(A_1, B_1), \Psi_-(A_2,B_2)]&=\frac{1}{4}[A_1+B_1+\varphi(A_1-B_1),A_2+B_2+\varphi(A_2-B_2)]\\
		&=\frac{1}{4}([A_1+B_1,A_2+B_2]+[\varphi(A_1-B_1),A_2+B_2]\\
		&\quad+[A_1+B_1,\varphi(A_2-B_2)]+[\varphi(A_1-B_1),\varphi(A_2-B_2)])\\
		&\overset{\eqref{id1}, \eqref{ident}}{=}\frac{1}{4}([A_1+B_1,A_2+B_2]+\varphi([A_1-B_1,A_2+B_2])\\
		&\quad+\varphi([A_1+B_1,A_2-B_2])+[A_1-B_1,A_2-B_2])\\
		&=\frac{1}{2}\left([A_1, A_2]+[B_1,B_2]+\varphi([A_1,A_2]-[B_1,B_2])\right)\\&=\Psi_-([A_1, A_2], [B_1,B_2])=\Psi_-([(A_1, B_1), (A_2,B_2)])
		\end{split}
		\end{equation*}
	and similarly
		\begin{equation*}
	\begin{split}
	[\Psi_+(A_1+i B_1), \Psi_+(A_2+iB_2)]&=[A_1+\varphi(B_1), A_2+\varphi(B_2)]=\\
	&=[A_1, A_2]+[\varphi(B_1), A_2]+[A_1, \varphi(B_2)]+[\varphi(B_1), \varphi(B_2)]\\
	&\overset{\eqref{id1},\eqref{ident}}{=}[A_1, A_2]+\varphi([B_1, A_2])+\varphi([A_1, B_2])-[B_1, B_2]\\
	&=\Psi_+([A_1, A_2]-[B_1, B_2]+i([A_1, B_2]+[B_1, A_2]))\\
	&=\Psi_+([A_1+i B_1, A_2+iB_2]).
		\end{split}
	\end{equation*}
Hence, if $G/H$ is of compact type,  then it is isometric to the irreducible type II  symmetric space $H\times H/H$ and if $G/H$  is of non-compact type, then it is isometric to the irreducible type IV symmetric space $H^{\mathbb{C}}/H$. 
\end{proof}

\section{The classification}

We now consider a manifold $(N,g_N, \tau)$ satisfying the conclusion of Theorem \ref{thred}. In order to keep notation as simple as possible, we denote by $g:=g_N$. Thus $(N,g)$ is a complete simply connected Riemannian manifold endowed with a metric connection 
$$\nabla=\nabla^g+\tau$$
with skew-symmetric torsion $\tau\in\Omega^3(\mathrm{T}N)$ satisfying 
\begin{equation}\label{paralleltorsion}
 \nabla^g\tau=0
\end{equation} 
and such that for all vectors $X\in \mathrm{T}N$:
\begin{equation}\label{tauXtau}
(\tau_X)_*\tau=0.
\end{equation} 

Since $\tau$ is $\nabla^g$-parallel, its kernel $\mathrm{Ker}(\tau_x):=\{X\in \mathrm{T}_x N\, \,|\, \tau_{X}=0\}\subset \mathrm{T}_xN$ 
defines a $\nabla^g$-parallel distribution on $N$. Hence, the manifold $(N,g)$ splits as a product $(N', g')\times (N'', g'')$, where $\tau$ acts trivially on $N'$ and its restriction to $N''$ still satisfies \eqref{paralleltorsion} and~\eqref{tauXtau} and has, moreover,  trivial kernel. 
In order to keep notation simple, we further denote $(N'',g'', \tau|_{N''})$ by $(N,g,\tau)$.
Let us denote by $\mathfrak{g}:=\mathrm{T}_xN$, for some fixed point $x\in N$.  The condition \eqref{tauXtau} implies that for all $X,Y\in \mathfrak{g}$ the following equality holds:
\begin{equation}\label{tauXtau1} 
[\tau_X, \tau_Y]=\tau_{\tau_X Y}.
\end{equation}

The bracket on $\mathfrak{g}$ defined by $[X,Y]:=\tau_X Y$, for all $X, Y\in\mathfrak{g}$ satisfies the Jacobi identity thanks to \eqref{tauXtau1}, and tautologically
$\tau$ becomes a morphism of Lie algebras from $\mathfrak{g}$ to $\Lambda^2 \mathrm{T}_x N$. Because $\tau$ is totally skew-symmetric, the metric $g$ on $\mathrm{T}_x N$ is $\mathrm{ad}(\mathfrak{g})$-invariant, so the Lie algebra $\mathfrak{g}$ is of compact type. Moreover, as $\tau$ has no kernel, $\mathfrak{g}$ is semisimple.

Consider the $g$-orthogonal splitting of $\mathfrak{g}=\mathrm{T}_x N$ into  simple Lie algebras: $\displaystyle\mathfrak{g}=\overset{\ell}{\underset{i=1}{\oplus}}\mathfrak{g}_i$.
If $\mathfrak{hol}$ denotes the holonomy algebra of the Levi-Civita connection $\nabla^g$ at the fixed point $x\in N$, then the following result holds:

\begin{lemma}\label{splitsimplealg} Each summand $\mathfrak{g}_i$ is $\mathfrak{hol}$-invariant.
\end{lemma}

\begin{proof}
Let $A\in \mathfrak{hol}$ and $X\in\mathfrak{g}_i$, for some $i\in\{1, \dots, \ell\}$. We need to show that $AX\in\mathfrak{g}_i$. Since $\tau$ is parallel, $\mathfrak{hol}$ acts trivially on $\tau$, so $A_*\tau_X=\tau_{AX}$. For any $j\in\{1, \dots, \ell\}\setminus\{i\}$ and $Y\in\mathfrak{g}_j$, we have $\tau(X,Y)=0$, so we obtain:
\begin{equation}\label{comm}
[AX, Y]=\tau_{AX}Y=A\tau_X Y-\tau_X AY=-\tau_X AY=-[X, AY].
\end{equation}
The left hand side of \eqref{comm}, namely $[AX, Y]$, belongs to $\mathfrak{g}_j$, whereas its right hand side belongs to $\mathfrak{g}_i$, showing that both sides have to vanish. Thus $[AX, Y]=0$, for all $Y\in\mathfrak{g}_i^\perp$, so $AX$ belongs to the commutator of $\mathfrak{g}_i^\perp$ in $\mathfrak{g}$, which coincides with the simple Lie algebra $\mathfrak{g}_i$. Hence, $AX$ belongs to $\mathfrak{g}_i$.
\end{proof}

Lemma~\ref{splitsimplealg}  implies that the tangent bundle of $N$ decomposes as a  sum of $\nabla^g$-parallel distributions, defined by the parallel transport of $\mathfrak{g}_i$, for $i\in\{1, \dots, \ell\}$.  Thus $(N,g)$ splits as a product $\displaystyle\prod_{i=1}^{\ell}(N_i, g_i)$ and $\tau=\overset{\ell}{\underset{i=1}{\sum}}\tau_i$, where each $\tau_i\in\Lambda^3(\mathrm{T}N_i)$ has trivial kernel, satisfies \eqref{paralleltorsion} and~\eqref{tauXtau}, and $\mathfrak{g_i}=\tau_i(\mathrm{T}_xN_i)$ is a simple Lie algebra.
Let us fix an $i\in\{1,\dots, \ell\}$ and again simplify the notation and  denote $(N_i,g_i,\tau_i)$ by $(N,g,\tau)$.

We are ready for the main step in the classification:

\begin{theorem} \label{thdec}
	Let $(N,g)$ be a complete simply connected Riemannian mani\-fold carrying a metric connection with parallel skew-symmetric torsion $\tau$ which satisfies $(\tau_X)_*\tau=0$, for all $X\in\Gamma(\mathrm{T}N)$,  $\mathrm{ker}(\tau)=0$ and $\mathfrak{g}:=\mathrm{T}_xN$ is a simple Lie algebra, for some $x\in N$. Then one of the following cases holds:
	\begin{enumerate}
		\item $(N,g)$ is an oriented $3$-dimensional Riemannian manifold and $\tau$ is a constant multiple of its Riemannian volume form $\mathrm{vol}_g$.
		\item $(N,g)$ is a simple Lie algebra with an $\mathrm{ad}$-invariant metric $g$ and $\tau$ is a constant multiple of its canonical $3$-form defined in \eqref{defcanform}.
		\item $(N,g)$ is an irreducible symmetric space of type II or of type IV and $\tau$ is a constant multiple of the $3$-form constructed in Example \ref{exsym}.
	\end{enumerate}	

\end{theorem}

\begin{proof}
Let us consider the de Rham decomposition $N=N_0\times N_1\times \cdots\times N_k$, where $N_0$ denotes the flat factor and each $N_{\alpha}$ is de Rham irreducible, for $\alpha\in\{1, \dots, k\}$. We may assume that $k\geq 1$, since otherwise $N=N_0$ is a simple Lie algebra, $g$ is an $\mathrm{ad}$-invariant metric and $\tau$ is proportional to its canonical $3$-form, which is case~$(2)$. 

We denote by $\mathfrak{hol}$ the holonomy algebra at the fixed point $x$ and by $D_\alpha:=\mathrm{T}_x N_{\alpha}$ the $\mathfrak{hol}$-invariant subspaces of $\mathrm{T}_x N$, for all $\alpha\in\{0, 1, \dots, k\}$.

Let $\alpha\in\{1, \dots, k\}$ be fixed. We denote by $V_1:=D_\alpha$, $V_2:=D_\alpha^{\perp}$, $\mathfrak{h}:=\mathfrak{hol}$, and by $\rho_i:\mathfrak{h}\to \mathfrak{so}(V_i)$ the restrictions of the holonomy representation to $V_i$ for $i=1,2$. The de Rham decomposition theorem gives the existence of an element $A\in\mathfrak{h}$ acting non-trivially on $V_1$ and trivially on $V_2$. From 
Lemma~\ref{lemmalg} (proved below in the Appendix), we obtain that  $\tau(D_\alpha, D_\alpha^{\perp}, D_\alpha^{\perp})=0$, whence
\begin{equation}\label{et}\tau(D_\alpha, D_\beta, D_\gamma)=0,\qquad\forall\ \alpha\in\{1, \dots, k\},\ \forall\ \beta,\gamma\in \{0, \dots, k\}\setminus\{\alpha\}.\end{equation}
%

From \eqref{et},  we immediately obtain:
\begin{equation}\label{eqtau}\tau \in\Lambda^3 D_0\oplus \overset{k}{\underset{\alpha=1}{\oplus}}(\Lambda^1 D_0\otimes\Lambda^2 D_\alpha)\oplus (\overset{k}{\underset{\alpha=1}{\oplus}}\Lambda^3 D_\alpha).\end{equation}
We consider the following two cases:

{\bf 1. $D_0$ is trivial.} If $k\geq 2$, then $\displaystyle\tau\in\overset{k}{\underset{\alpha=1}{\oplus}} \Lambda^3 D_\alpha$, which implies that each $D_\alpha$ is a Lie subalgebra of $\mathfrak{g}=\overset{k}{\underset{\alpha=1}{\oplus}} D_\alpha$, contradicting the assumption that $\mathfrak{g}$ is simple. Therefore $k=1$ and the above splitting has only one non-trivial component, meaning that $(N,g)$ is de Rham irreducible. Let us denote by $n:=\mathrm{dim}(N)=\mathrm{dim}(\mathfrak{g})$. By the Berger-Simons holonomy theorem, $(N,g)$ is either an irreducible symmetric space, or its holonomy belongs to the list of Berger: $\mathrm{SO}(n)$, $\mathrm{U}(n/2)$,
$\mathrm{SU}(n/2),\ \mathrm{Sp}(n/4),\ \mathrm{Sp}(n/4)\cdot\mathrm{Sp}(1),\ \mathrm{G}_2$ for $n=7$, or $\mathrm{Spin}(7)$ for $n=8$ (see \cite{besse}, p. 301).

In the former case, Theorem~\ref{symmspaces} implies that $(N,g)$ is an irreducible symmetric space of type II and IV, so we are in case $(3)$. 

In the latter case, we remark that for $n\ge 8$, all holonomy groups in the list of Berger have dimension strictly larger than $n$. On the other hand, the hypothesis $\mathrm{ker}(\tau)=0$ together with the inclusion $\mathfrak{hol}\subseteq\tau(\mathfrak{g})$ proved in Lemma~\ref{alglemma} below, show that $\mathrm{dim}(\mathfrak{hol})\le n$. Thus $n\le 7$, and since $\mathfrak{g}$ is a simple Lie algebra of dimension $n$, the only possible case is $n=3$. Then the parallel $3$-form $\tau$ has to be a constant multiple of the Riemannian volume form of $(N,g)$, so we are in case $(1)$.

{\bf 2. $D_0$ is not trivial.} We will show that $k=1$. If we denote by $r:=\dim(D_0)\geq 1$, then using \eqref{eqtau}, the $3$-form $\tau$ can be decomposed as follows:
\begin{equation}\label{tau}\tau=\eta+\sum_{i=1}^{r}\sum_{\alpha=1}^{k}\xi_i\wedge \omega_{i\alpha}+\sum_{\alpha=1}^{k}\sigma_\alpha,\end{equation}
where $\eta\in \Lambda^3 D_0$, $\{\xi_i\}_{i=\overline{1,r}}$ is a basis of $D_0$, and $\omega_{i\alpha}\in\Lambda^2 D_\alpha$, $\sigma_\alpha\in\Lambda^3 D_\alpha$, for $i\in\{1, \dots, r\}$ and $\alpha\in\{1, \dots, k\}$. Note that the holonomy algebra $\mathfrak{hol}$ acts trivially on the two-forms $\omega_{j\alpha}\in\Lambda^2 D_\alpha$ so by the $\mathfrak{hol}$-irreducibility of $D_\alpha$, each $\omega_{j\alpha}$ is proportional to a complex structure on $D_\alpha$. 

For any $j\in \{1, \dots, r\}$ we have $\displaystyle\tau_{\xi_j}=\eta_{\xi_j}+\sum_{\alpha=1}^{k} \omega_{j\alpha}$. 
Thus the condition that $(\tau_X)_*\tau=0$ applied to $X=\xi_j$ yields:
\[0= (\tau_{\xi_j})_*\tau=(\eta_{\xi_j})_*\eta+\sum_{i=1}^{r}\sum_{\alpha=1}^{k}\eta_{\xi_j}\xi_i\wedge \omega_{i\alpha}+\sum_{i=1}^{r}\sum_{\alpha=1}^{k}\xi_i\wedge[\omega_{j\alpha},\omega_{i\alpha}]+\sum_{\alpha=1}^{k}(\omega_{j\alpha})_*\sigma_\alpha,\]
which, by comparing the types of the forms,  is equivalent to the following set of identities, for all $j,\ell\in \{1, \dots, r\}$ and $\alpha\in\{1, \dots, k\}$:
\begin{equation}\label{ident3}
\begin{cases}
(\eta_{\xi_j})_*\eta=0,\\[0.2cm]
\displaystyle [\omega_{j\alpha}, \omega_{\ell\alpha}]+\sum_{i=1}^r\eta(\xi_j,\xi_i, \xi_{\ell})\omega_{i\alpha}=0,\\[0.2cm]
(\omega_{j\alpha})_*\sigma_{\alpha}=0.
\end{cases}
\end{equation}

We first prove that all $3$-forms $\sigma_\alpha$ vanish, for $\alpha\in\{1, \dots, k\}$. Assuming, by contradiction, that there exists some $\alpha\in \{1, \dots, k\}$ with $\sigma_\alpha\neq 0$, then the last identity in \eqref{ident3} implies that $\omega_{j\alpha}=0$, for all $j\in \{1, \dots, r\}$ (indeed, the action of complex structures on odd-dimensional exterior forms is injective). 
Hence, by \eqref{tau}, $\tau\in\Lambda^3 D_{\alpha}\oplus\Lambda^3 D_{\alpha}^{\perp}$, which yields that $\mathfrak{g}$ splits in a direct sum of Lie subalgebras, $\mathfrak{g}=D_\alpha\oplus D_\alpha^{\perp}$, contradicting the assumption that $\mathfrak{g}$ is simple. This shows that $\sigma_\alpha=0$, for all $\alpha\in\{1, \dots, k\}$. Consequently, \eqref{tau} reads:
\begin{equation}\label{tau1}\tau=\eta+\sum_{i=1}^{r}\sum_{\alpha=1}^{k}\xi_i\wedge \omega_{i\alpha}.\end{equation}

The first identity in \eqref{ident3} ensures that $\mathfrak{g}_0:=D_0=\mathrm{T}_x N_0$ carries a Lie algebra structure, with bracket $[X,Y]:=\eta_XY$, such that $\eta$ is a Lie morphism from $\mathfrak{g}_0$ to $\Lambda^2 D_0$.
The second set of identities in \eqref{ident3} shows that for each $\alpha\in\{1, \dots, k\}$, the linear map
\[\rho_\alpha\colon \mathfrak{g}_0\to \Lambda^2 D_\alpha, \quad \rho_\alpha(\xi_i):=\omega_{i\alpha}, \text{ for } i\in\{1, \dots, r\}\]
is a representation of the Lie algebra $\mathfrak{g}_0$ on $D_\alpha$.
We next prove that 
\begin{equation}\label{split}
\tau\in \Lambda^3(\mathrm{im}(\rho_1^*)\oplus D_1)\oplus \Lambda^3\left(\mathrm{ker}(\rho_1)\oplus \left(\overset{k}{\underset{\alpha=2}{\oplus}}D_{\alpha}\right)\right).
\end{equation}

For this, we consider the decomposition $\mathfrak{g}_0=\mathrm{ker}(\rho_1)\oplus\mathrm{im}(\rho_1^*)$ and we claim that the following inclusions hold: 
\begin{equation}\label{incl}
\mathrm{im}(\rho_1^*)\subseteq\mathrm{ker}(\rho_\alpha), \quad \forall \alpha\in\{2, \dots, k\}.
\end{equation}
Indeed, if $\{e_i\}_i$ is a basis of $\mathrm{T}_x N$, the $4$-form $\displaystyle\sum_{i=1}^n\tau_{e_i}\wedge\tau_{e_i}$, vanishes by Lemma~\ref{elemprop}. In particular, for any $\alpha\in\{2, \dots, k\}$  the projection of this form onto  $\Lambda^2(D_1)\otimes\Lambda^2(D_\alpha)$ vanishes, which can be written as $$ 0=\sum_{j=1}^r\omega_{j1}\wedge\omega_{j\alpha}=\sum_{j=1}^r\rho_{1}(\xi_j)\wedge\rho_{\alpha}(\xi_j).$$
Taking, for any $\xi\in\mathfrak{g}_0$, the interior product (the metric adjoint of the wedge product) with $\rho_1(\xi)$, we obtain that $$ \sum_{j=1}^r\<\rho_1(\xi), \rho_{1}(\xi_j)\>\rho_\alpha(\xi_j)=0,$$ 
which means  that $\displaystyle\rho_1^*\rho_1(\xi)= \sum_{j=1}^r\<\rho_1(\xi), \rho_{1}(\xi_j)\>\xi_j\in\mathrm{ker}(\rho_\alpha)$. As $\mathrm{im}(\rho_1^*\rho_1)=\mathrm{im}(\rho_1^*)$, we obtain that $\mathrm{im}(\rho_1^*)\subseteq\mathrm{ker}(\rho_\alpha)$, thus proving our claim \eqref{incl}.

Since $\rho_1$ is a representation, $\mathrm{ker}(\rho_1)$ is an ideal of $\mathfrak{g}_0$. Moreover, since the metric on $\mathfrak{g}_0$ is ad-invariant, its orthogonal complement $\mathrm{im}(\rho_1^*)$ is an ideal too. Thus the canonical 3-form $\eta$ of the metric Lie algebra $\mathfrak{g}_0$ decomposes as $\eta=\eta_{1}+\eta_{2}$, with $\eta_{1}\in\Lambda^3 (\mathrm{im}(\rho_1^*))$, and $\eta_{2}\in\Lambda^3 (\mathrm{ker}(\rho_1))$.
Let $\{\xi_i\}_{i=\overline{1,d_1}}$ and $\{\zeta_j\}_{j=\overline{1,d_2}}$ be orthonormal bases of $\mathrm{im}(\rho_1^*)$ and $\mathrm{ker}(\rho_1)$ respectively.
According to \eqref{tau1} and using the inclusions \eqref{incl}, the $3$-form $\tau$ can be decomposed as follows:
\[\tau=\underbrace{\eta_{1}+\sum_{i=1}^{d_1}\xi_{i}\wedge\rho_{1}(\xi_{i})}_{\in \Lambda^3(\mathrm{im}(\rho_1^*)\oplus D_1)}+\underbrace{\eta_{2}+\overset{k}{\underset{{\alpha=2}}{\sum}}\sum_{j=1}^{d_2}\zeta_{j}\wedge\rho_{\alpha}(\zeta_{j})}_{\in\Lambda^3\left(\mathrm{ker}(\rho_1)\oplus\left(\overset{k}{\underset{\alpha=2}{\oplus}}D_{\alpha}\right)\right)},\]
thus proving the splitting \eqref{split}.
Since $\mathfrak{g}$ is simple and $D_1\neq 0$, it follows that $\mathrm{ker}(\rho_1)=0$ and $\overset{k}{\underset{\alpha=2}{\oplus}}D_{\alpha}=0$, whence $k=1$, $\mathfrak{g}=\mathfrak{g}_0\oplus D_1$, and $N=N_0\times N_1$. 

We now use again the inclusion $\mathfrak{hol}\subseteq \tau(\mathfrak{g})$ provided by Lemma~\ref{alglemma} in the Appendix. Since $\mathfrak{hol}$ preserves $D_1$ and $\displaystyle\tau=\eta+\sum_{i=1}^r \xi_i\wedge \omega_{i1}\in\Lambda^3D_0\oplus (\Lambda^1D_0\otimes\Lambda^2D_1)$, this implies that each element of the holonomy algebra is of the form $\tau(\xi)$, for some $\xi\in\mathfrak{g}_0$, \emph{i.e.} $\mathfrak{hol}\subseteq \tau(\mathfrak{g}_0)$. On the other hand, for all $j\in\{1, \dots, r\}$, $\omega_{j1}$ is a parallel $2$-form, so $[\mathfrak{hol}, \tau(\mathfrak{g}_0)]=0$, which shows in particular that $\mathfrak{hol}$ is a commutative Lie algebra. As the holonomy representation of $\mathfrak{hol}$ on $D_1$ is irreducible, the commutativity of $\mathfrak{hol}$ implies that $\dim(D_1)=2$ and $\dim(\mathfrak{g}_0)=1$, because $\rho_1\colon \mathfrak{g}_0\to\Lambda^2 D_1$ is injective. Thus $N$ is $3$-dimensional, which is again case $(3)$.
\end{proof}

Summing up, we have shown the following result:

\begin{theorem}\label{thmain}
Let $\xi$ be a non-zero vector field and let $\nu$ be a $3$-form on a complete simply connected Riemannian manifold $(M,g_M)$. Then the metric connection $\nabla_X:=\nabla^{g_M}_X+X\wedge \xi+\nu_X$ has $\nabla$-parallel twistor-free torsion if and only if $(M, g_M)$ 
is homothetic to a warped product $(N\times \mR, e^{2t}g_N+\di t^2)$, with $\xi=\frac{\partial}{\partial t}$ and $\nu=e^{3t}\tau$, where $(N, g_N,\tau)$ is a Riemannian product of complete simply connected Riemannian manifolds $(N_i,g_i)$ endowed with $3$-forms $\tau_i$, such that each  $(N_i,g_i,\tau_i)$ is of one of the following types:
\begin{enumerate}
		\item $(N_i,g_i)$ is a $3$-dimensional oriented Riemannian manifold and $\tau_i$ is a constant multiple of its Riemannian volume form.
		\item $N_i$ is a simple Lie algebra with an $\mathrm{ad}$-invariant metric $g_i$ and $\tau_i$ is a constant multiple of its canonical $3$-form defined in \eqref{defcanform}.
		\item $(N_i,g_i)$ is an irreducible symmetric space of type II or of type IV and $\tau_i$ is a constant multiple of the $3$-form constructed in Example \ref{exsym}.
		\item $(N_i,g_i)$ is a Riemannian manifold and $\tau_i=0$.
\end{enumerate}	

\end{theorem}

\section{Appendix. Some representation theory}

We finally prove the two representation theoretical results that have been used above. 

\begin{lemma}\label{alglemma}
	Let $\mathfrak{h}$ be a Lie algebra and let $\rho\colon \mathfrak{h}\to\mathfrak{so}(V)\simeq \Lambda^2V$ be an orthogonal representation of $\mathfrak{h}$  on a finite dimensional Euclidean vector space $(V, \<\cdot, \cdot\>)$. Assume that $\tau\in\Lambda^3 V$ is an $\mathfrak{h}$-invariant $3$-form, such that $\tau_X\ne 0$ for all $X\in V\setminus\{0\}$ and $(\tau_X)_*\tau=0$ for all $X\in V$. Then the following inclusion holds:
	\[\rho(\mathfrak{h})\subseteq  \tau(V).\]
\end{lemma}	

\begin{proof}
	Let us denote by $W$ the orthogonal complement of $\tau(V)$ in $(\tau(V)+\rho(\mathfrak{h}))$. We claim that $W=\{0\}$.
	
	Let $A\in W$. 
	We first notice that $A_*\tau=0$, because $A\in \tau(V)+\rho(\mathfrak{h})$ and $\tau$ is $\mathfrak{h}$-invariant and satisfies also $(\tau_X)_*\tau=0$, for all $X\in V$.	Then for any $X\in V$ we have:
\begin{equation}\label{a}A_*\tau_X=\tau_{AX}.
\end{equation}
	
	For any $Y\in V$ we compute 
	$$\<A_*\tau_X, \tau_Y\>=\<[A,\tau_X], \tau_Y\>=\<A, [\tau_X, \tau_Y]\>=\<A, \tau_{\tau_X Y}\>=0,$$
	where for the last equality we used that $A\in \tau(V)^\perp$. Thus, we obtained that $A_*\tau_X\in \tau(V)^\perp$, which together with \eqref{a} implies that 
	$0=A_*\tau_X=\tau_{AX}$. 
	 The hypothesis on $\tau$ yields $AX=0$, for all $X\in V$, and thus $A=0$. 
	 
	 This shows that $W=\{0\}$, whence $\tau(V)=\tau(V)+\rho(\mathfrak{h})$ and thus $\rho(\mathfrak{h})\subseteq  \tau(V)$.
\end{proof}	

The second result is an avatar of Lemma 3.4 in \cite{cms}:

\begin{lemma}\label{lemmalg}
	Let $\mathfrak{h}$ be a Lie algebra  of compact type and let $\rho_j\colon \mathfrak{h}\to \mathfrak{so}(V_j)$, for $j\in\{1,2\}$, be two orthogonal representations of $\mathfrak{h}$, such that $V_1$ is irreducible and there exists  $A\in\mathfrak{h}$ with $\rho_1(A)\neq 0$ and $\rho_2(A)=0$. If $\tau$ is an $\mathfrak{h}$-invariant $3$-form on $V_1\oplus V_2$, then $\tau(X_1, Y_2, Z_2)=0$, for every $X_1\in V_1$, and $Y_2, Z_2\in V_2$.
\end{lemma}

\begin{proof}
The subspace $V:=\rho_1(\mathrm{ker}(\rho_2))(V_1)$ of $V_1$ is non-zero since it contains the image of $\rho_1(A)$. We claim that $V$ is  $\mathfrak{h}$-invariant. Indeed, $\mathrm{ker}(\rho_2)$ is an ideal in $\mathfrak{h}$, so for all $C\in\mathfrak{h}$, $B\in\mathrm{ker}(\rho_2)$ and $X_1\in V_1$ we have $[C,B]\in \mathrm{ker}(\rho_2)$ and thus
	\[\rho_1(C)\rho_1(B)X_1=\rho_1([C,B])X_1+\rho_1(B)\rho_1(C)X_1\in V_1.\]
	
	Therefore $V=V_1$ by the irreducibility of $\rho_1$. Let $X_1\in V_1$ and $Y_2, Z_2\in V_2$. Since $V=V_1$, there exists $X_1'\in V_1$ and $B\in\mathrm{ker}(\rho_2)$, such that $X_1=\rho_1(B)X_1'$. Using the $\mathfrak{h}$-invariance of $\tau$, we compute:
	\[\tau(X_1, Y_2, Z_2)= \tau(\rho_1(B)X_1', Y_2, Z_2)=- \tau(X_1', \rho_2(B)Y_2, Z_2)-\tau(X_1', Y_2, \rho_2(B)Z_2)=0.\]
\end{proof}

\end{document}